\documentclass{amsart}                 % final size

\usepackage{amssymb,amsmath,amsthm}
\usepackage{hyperref}
\usepackage{color}

\usepackage{ulem}
\renewcommand{\emph}{\normalem}

\theoremstyle{plain}
\newtheorem{theorem}{Theorem}[section]
\newtheorem*{theorem*}{Theorem}
\newtheorem{lemma}[theorem]{Lemma}
\newtheorem*{lemma*}{Lemma}

\newtheorem{proposition}[theorem]{Proposition}

\theoremstyle{remark}
\newtheorem{remark}[theorem]{Remark}
\newtheorem*{remark*}{Remark}

\theoremstyle{definition}
\newtheorem{definition}[theorem]{Definition}
\newtheorem*{definition*}{Definition}

\theoremstyle{notation}%%%%%%%%%%%
\newtheorem{notation}[theorem]{Notation}
\newtheorem*{notation*}{Notation}

\newtheorem{Basic assumptions}[theorem]{Basic assumptions}

\numberwithin{equation}{section}

\newcount\quantno
\everydisplay{\quantno=0}\everycr{\quantno=0}
\newcommand\quant{\advance\quantno by1
                      \ifnum\quantno=1\qquad\else\quad\fi\forall }
\newcommand\itemno[1]{(\romannumeral #1)}

\renewcommand\Im{\operatorname{\mathrm{Im}}}

\newcommand\Dom{\mathrm{Dom}}

\newcommand\rest[1]{\kern-.1em
          \lower.5ex\hbox{$\scriptstyle #1$}\kern.05em}

\newcommand\set[1]{{\left\{#1\right\}}}

\renewcommand\mod[1]{\vert{#1}\vert}
\newcommand\bigmod[1]{\bigl\vert{#1}\bigr|}

\newcommand\norm[2]{{\Vert{#1}\Vert_{#2}}}

\newcommand\bignorm[2]{\left.{\bigl\Vert{#1}\bigr\Vert_{#2}}\right.}

\newcommand\opnorm[2]{|\!|\!| {#1} |\!|\!|_{#2}}
\newcommand\bigopnorm[2]{\bigl|\!\bigl|\!\bigl| {#1} 
\bigr|\!\bigr|\!\bigr|_{#2}}

\newcommand\wrt{\,\text{\rm d}}

\newcommand\bS{\mathbf{S}}

\newcommand\BC{\mathbb{C}}

\newcommand\BN{\mathbb{N}}
\newcommand\BR{\mathbb{R}}

\newcommand\BZ{\mathbb{Z}}

\newcommand\cB{\mathcal{B}}

\newcommand\cD{\mathcal{D}}

\newcommand\cI{\mathcal{I}}   
 
\newcommand\cJ{\mathcal{J}}
    
\newcommand\cL{\mathcal{L}}

\newcommand\cO{\mathcal{O}}

\newcommand\cR{\mathcal{R}}    
\newcommand\cS{\mathcal{S}}  \newcommand\fS{\mathfrak{S}}  
\newcommand\cT{\mathcal{T}}    
\newcommand\cU{\mathcal{U}}    
    
\newcommand\cW{\mathcal{W}}

\newcommand\al{\alpha}
\newcommand\be{\beta}
    
\newcommand\de{\delta}
  \newcommand\vep{\varepsilon}

\newcommand\la{\lambda}   
\newcommand\om{\omega}    \newcommand\Om{\Omega}  
\newcommand\si{\sigma}
\newcommand\te{\theta}

\newcommand\OV{\overline}
\newcommand\funnyk{k\hbox to 0pt{\hss\phantom{g}}}

\newcommand\lu[1]{L^1(#1)}
\newcommand\lp[1]{L^p(#1)}

\newcommand\ld[1]{L^2(#1)}

\newcommand\ldO[1]{L^2_0(#1)}

\newcommand\ly[1]{L^\infty(#1)}

\newcommand\hu[1]{H^1(#1)}
\newcommand\hufin[1]{H_{\mathrm{fin}}^1(#1)}

\newcommand\Xu[1]{X^1(#1)}

\newcommand\Xhfin[1]{X_{\mathrm{fin}}^k(#1)}

\newcommand\Xh[1]{X^k(#1)}
\newcommand\Xk{{X^k}}
\newcommand\Xhmu[1]{X^{k-1}(#1)}
\newcommand\Xhat[1]{X_{\mathrm{at}}^k(#1)}
\newcommand\Xkat{X_{\mathrm{at}}^k}

\newcommand\Yu[1]{Y^1(#1)}
\newcommand\Yh[1]{Y^k(#1)}
\newcommand\Yk{Y^k}

\newcommand\wh{\widehat}

\newcommand\whH{\widehat{\phantom{G}}\hbox to 0pt{\hss $H$}}

\newcommand\emspace{\hbox to 6pt{\hss}}
\newcommand\ds{\displaystyle}

\newcommand\rmi{\hbox{\rm (i)}}
\newcommand\rmii{\hbox{\rm (ii)}}
\newcommand\rmiii{\hbox{\rm (iii)}}

\newcommand\ir{\int_{-\infty}^{\infty}}

\newcommand\One{{\mathbf{1}}}

\newcommand\e{\mathrm{e}}

\newcommand\dest{\text{\rm d}}

\newcommand\Ric{\mathop{\rm Ric}}

\newcommand\Rbekappa{\cR_{4\be^2+\kappa^2}}
\newcommand\Jbe{\cU_{\be^2}}
\newcommand\Jbekappa{\cU_{4\be^2+\kappa^2}}

\newcommand\Jbeh{\cU_{\be^2}^k}
\newcommand\Jbemenoh{\cU_{\be^2}^{-k}}

\newcommand\HBh{Q_B^k}
\newcommand\HBhperp{(Q_B^k)^{\perp}}

\newcommand\sgn{\mathrm{sgn}}

\begin{document}

\title[Atomic decomposition]
{Atomic decomposition \\ of Hardy type spaces \\
on certain noncompact manifolds}

\subjclass[2000]{} 


\thanks{Work partially supported by PRIN 2007 ``Analisi Armonica".}

\author[G. Mauceri, S. Meda and M. Vallarino]
{Giancarlo Mauceri, Stefano Meda and Maria Vallarino}

\address{Giancarlo Mauceri: Dipartimento di Matematica\\ 
Universit\`a di Genova\\
via Dodecaneso 35\\ 16146 Genova\\ Italy 
-- mauceri@dima.unige.it}

\address{Stefano Meda: 
Dipartimento di Matematica e Applicazioni
\\ Universit\`a di Milano-Bicocca\\
via R.~Cozzi 53\\ I-20125 Milano\\ Italy
-- stefano.meda@unimib.it}

\address{Maria Vallarino:
Dipartimento di Matematica e Applicazioni
\\ Universit\`a di Milano-Bicocca\\
via R.~Cozzi 53\\ I-20125 Milano\\ Italy
 --  maria.vallarino@unimib.it}

\begin{abstract}
In this paper we consider a complete connected noncompact 
Riemannian manifold $M$ with bounded geometry and spectral gap.  
We prove that the Hardy type 
spaces $\Xh{M}$, introduced in a previous paper of the authors,
%\colblu{\cite{MMV2}}, 
have an atomic characterization. An
 atom in $\Xh{M}$ is an atom in the Hardy space
$\hu{M}$ introduced by Carbonaro, Mauceri and Meda, satisfying
an ``infinite dimensional'' cancellation condition.
As an application, we prove that  the Riesz
transforms of even order $\nabla^{2k} \cL^{-k}$ 
map $\Xh{M}$ into $\lu{M}$.
\end{abstract}

\maketitle

\setcounter{section}{0}
\section{Introduction} \label{s:Introduction}

Suppose that $M$ is a complete connected noncompact 
Riemannian manifold with Ricci curvature bounded from below
and positive injectivity radius.
Denote by~$-\cL$ the Laplace--Beltrami operator on $M$: 
$\cL$ is a symmetric operator on $C_c^\infty(M)$ (the space of
compactly supported smooth complex-valued functions on $M$).  Its closure
is a self adjoint operator on $\ld{M}$ which, with a slight
abuse of notation, we still denote by $\cL$. 
We assume throughout that the bottom $b$ of the 
spectrum of $\cL$ is \emph{strictly positive}. 
Important examples of manifolds with these properties are nonamenable
connected unimodular Lie groups equipped with a left
invariant Riemannian distance, and symmetric spaces
of the noncompact type with the Killing metric. 
It is known \cite[Section~8]{CMM1}
that for manifolds with Ricci curvature bounded
from below the assumption $b>0$ is equivalent to 
an isoperimetric property,  
which implies that $M$ has exponential volume growth,
\emph{ergo} {the Riemannian measure} is nondoubling.

In \cite{MMV2} we introduced a sequence $\Xu{M}, X^2(M), \ldots$
of new spaces of Hardy type on $M$ with the property that
$$
\hu{M} \supset \Xu{M} \supset  X^2(M), \ldots,
$$
and the sequence $\Yu{M}, Y^2(M), \ldots$ 
of their dual spaces, and showed that these spaces may 
be used to obtain endpoint estimates for 
interesting spectral multipliers of $\cL$, including the
purely imaginary powers of $\cL$, and the first order Riesz transform.
Here $\hu{M}$ is the atomic Hardy space introduced in \cite{CMM1}.
Each of the inclusions above is proper and each of 
the spaces $\Xh{M}$ is an isometric copy of $\hu{M}$. 
We refer the reader to Section~\ref{s: Background material} for 
the definitions of the spaces $\hu{M}$, $\Xh{M}$ and $\Yh{M}$.  

Since $\Xh{M}$ is continuously included in $\hu{M}$, 
each function in $\Xh{M}$ admits an atomic
decomposition in terms of $H^1$-atoms (these are defined 
as classical Euclidean atoms {\cite{CW,St2}}, but their 
support is contained in balls of radius at most {$1$})).
Recall that an atom $a$ in $\hu{M}$ must have integral $0$.
This cancellation condition may
also be expressed by saying that $a$ is orthogonal to the subspace
of $\ld{M}$ of functions that are constant on 
the support of $a$.

E.M.~Stein posed the question whether functions in $\Xh{M}$
may be characterised as those functions in $\hu{M}$
that admit a decomposition in terms of atoms satisfying
further cancellation conditions. 
The purpose of this paper is to 
prove such an atomic characterisation of $\Xh{M}$ on manifolds
as above satisfying, in the case where {$k\geq 1$}, the additional 
requirement that the first {$k$} covariant derivatives of the 
Ricci tensor of $M$ be uniformly bounded.  

Specifically,
we say that $A$ is an  $X^k$-atom if $A$ is an $H^1$-atom supported in
a ball $B$ of radius at most $1$ and is orthogonal 
in $\ld{B}$ to the space $\HBh$ 
of all functions $V$ in $\ld{M}$ such that $\cL^k V$ is constant
on a neighbourhood of $\OV{B}$.
Note that, contrary to the classical case,
the cancellation condition required for $\Xh{M}$-atoms
is expressed as orthogonality to a 
\emph{infinite dimensional} subspace of $\ld{M}$.
{As far as we know, this is the first time that 
an ``infinite dimensional'' cancellation condition 
appears in the literature in connection with Hardy spaces.}

{An interesting and challenging problem is to prove
$\lp{M}$ bounds for the Riesz transforms for $p$ in $(1,\infty)$ 
and endpoint results for $p=1$.  After the pioneering
works of Stein \cite{St1} and R.S.~Strichartz \cite{Str}, 
several contributions
have appeared recently on the subject.  
We refer the reader to \cite{CD,ACDH} and
the references therein for $\lp{M}$ bounds.  
Endpoint results in the case where $\mu$
is doubling and $M$ satisfies 
some extra assumptions, such as 
appropriate on-diagonal estimate for the heat kernel 
or scaled Poincar\'e inequality~
have been obtained in \cite{CD,Ru,MRu,AMR}.

To the best of our knowledge, very little is known
about $\lp{M}$ bounds for higher order Riesz transforms. 
N.~Lohou\'e \cite{Lo} proved that if $M$ is a Cartan--Hadamard
manifold such that the first $2k$ covariant derivatives of the 
Riemann tensor of $M$ are uniformly bounded, and
the Laplace--Beltrami operator has spectral gap, then 
the Riesz transforms of even order $\nabla^{2k}\cL^{-k}$
are bounded on $\lp{M}$ for every $p$ in $(1,\infty)$.
}
The atomic characterization of the spaces
$\Xh{M}$ enables us to prove, in a more general setting,
an endpoint result for $\nabla^{2k}\cL^{-k}$ when $p=1$, 
namely that these operators are bounded from $\Xh{M}$ to $\lu{M}$
(see Theorem~\ref{t: RT}).  We then obtain the
$\lp{M}$ boundedness for $p$ in $(1,2)$ by interpolation with 
a classical $\ld{M}$ result of T.~Aubin~\cite{Au}.
We emphasise the fact that our proof is very short and simple. 

Now we briefly outline the content of this paper.
In Section~\ref{s: Background material}, after stating the basic 
geometric assumptions  on the manifold $M$ and 
their analytic consequences, we recall the definition 
of the spaces $\Xh{M}$ and their properties. In Section~\ref{s: 
Special atoms} we define the atoms in $\Xh{M}$, 
we prove some of their properties and we define the atomic 
space $\Xhat{M}$. In Section~\ref{s: The main result}, 
we prove that $\Xh{M}=\Xhat{M}$, with equivalent norms. The argument uses 
a technical lemma, whose proof is rather long and 
is deferred to {Section~\ref{s: proof of Lemma}}. In Section~\ref{s: Riesz} 
we prove the boundedness {results for} the 
Riesz transforms of even order. 

We will use the ``variable constant convention'', and denote by $C,$
possibly with sub- or superscripts, a constant that may vary from place to 
place and may depend on any factor quantified (implicitly or explicitly) 
before its occurrence, but not on factors quantified afterwards. 
If $\cT$ is a bounded linear operator from the Banach 
space $A$ to the Banach space $B$, we 
shall denote by $\bigopnorm{\cT}{A;B}$ its norm. 
If $A=B$ we shall simply write $\bigopnorm{\cT}{A}$ 
instead of  $\bigopnorm{\cT}{A;A}$.

\section{Basic definitions and background material}
\label{s: Background material}

Suppose that $M$ is a connected $n$-dimensional Riemannian manifold
of infinite volume with Riemannian measure $\mu$.
Denote by $\Ric$ the Ricci tensor, by $-\cL$ the Laplace--Beltrami operator
on $M$, 
by $b$ the bottom of the $\ld{M}$ spectrum of $\cL$,
and set $\be =
\limsup_{r\to\infty} \bigl[\log\mu\bigl(B(o,r)\bigr)\bigr]/(2r)$.
By a result of  {R.}~Brooks $b\leq \be^2$ \cite{Br}.

\begin{definition} \label{def: bounded geometry}
We say that $M$ has $C^\ell$ \emph{bounded geometry} 
if the injectivity radius is positive and the following hold:
\begin{itemize}
\item{} 
if $\ell =0$, then  the Ricci tensor is bounded from below;
\item{} 
if $\ell$ is positive, then the covariant
derivatives $\nabla^j \Ric$ of the Ricci tensor are uniformly
bounded on $M$ for all $j$ in $\{0,\ldots, \ell\}$.
\end{itemize}
\end{definition}

\begin{Basic assumptions} \label{Ba: on M}
We make the following assumptions on $M$:
\begin{enumerate}
\item[\itemno1] $b>0$;
\item[\itemno2]
$M$ has $C^\ell$ bounded geometry for some nonnegative integer $\ell$.
\end{enumerate}
\end{Basic assumptions}
We denote by~$\kappa$ the smallest positive number such that 
$\Ric \geq -\kappa^2$.
\begin{remark} \label{rem: unif ball size cond}
%%%%%%%%%%%%%
It is well known that for manifolds with properties \rmi-\rmii\ above
there are positive constants $\al$, $\be$ and $C$ such that
%%%%%%%%%%%%
\begin{equation} \label{f: volume growth} 
\mu\bigl(B(p,r)\bigr)
\leq C \, r^{\al} \, \e^{2\be \, r}
\quant r \in [1,\infty) \quant p \in M,
\end{equation}
where $B(p,r)$ denotes the  
geodesic ball with centre $p$ and radius~$r$.
\par
Moreover, they
satisfy the \emph{uniform ball size condition}, i.e., for every $r>0$
\begin{equation}\label{f: ubsc}
\inf\, \bigl\{ \mu\bigl(B(p, r)\bigr): p \in M \bigr\} > 0
\qquad\hbox{and}\qquad
\sup\, \bigl\{ \mu\bigl(B(p, r)\bigr): p \in M \bigr\} < \infty.
\end{equation}
See, for instance, \cite{CMP}, where complete references are given.
\end{remark}
\begin{remark}\label{ultra}
By \cite[Section~7.5]{Gr} there exists 
a nonnegative number $\de$ such that the following 
ultracontractive estimate holds
{
$$
\bigopnorm{\e^{-t\cL }}{1;2}
\leq C\, \e^{-bt} \, t^{-n/4} \, (1+t)^{n/4-\de/2}  \quant t \in \BR^+ .
$$
Clearly this implies
$$
\bigopnorm{\e^{-t\cL }}{1;\infty}
\leq C\, \e^{-bt} \, t^{-\de}  \quant t \in [1,\infty).
$$
}
\end{remark}
\medskip
We denote by $\cB$ the family of all balls on $M$.
For each $B$ in $\cB$ we denote by $c_B$ and $r_B$
the centre and the radius of $B$ respectively.  
Furthermore, we denote by $c \, B$ the
ball with centre $c_B$ and radius $c \, r_B$.
For each \emph{scale parameter} $s$ in $\BR^+$, 
we denote by $\cB_s$ the family of all
balls $B$ in $\cB$ such that $r_B \leq s$.

We recall the definitions of the atomic Hardy space $H^1(M)$ and its
dual space $BMO(M)$ given in \cite{CMM1}.  

\begin{definition}
An $H^1$-\emph{atom} $a$
is a function in $\lu{M}$ supported in a ball $B$
with the following properties:
\begin{enumerate}
\item[\itemno1]
$\int_B a \wrt \mu  = 0$;
\item[\itemno2]
$\norm{a}{2}  \leq \mu (B)^{-1/2}$.
\end{enumerate}
\end{definition}

\begin{definition}
Suppose that $s$ is in $\BR^+$.  
The \emph{Hardy space} $H_s^{1}({M})$ is the 
space of all functions~$g$ in $\lu{M}$
that admit a decomposition of the form
\begin{equation} \label{f: decomposition}
g = \sum_{k=1}^\infty \la_k \, a_k,
\end{equation}
where $a_k$ is a $H^1$-atom \emph{supported in a ball $B$ of $\cB_s$},
and $\sum_{k=1}^\infty \mod{\la_k} < \infty$.
The norm $\norm{g}{H_s^{1}}$
of $g$ is the infimum of $\sum_{k=1}^\infty \mod{\la_k}$
over all decompositions (\ref{f: decomposition})
of $g$.  
\end{definition}

The vector space $H_s^{1}(M)$ 
is independent of~$s$ in $\BR^+$.
Furthermore, given $s_1$ and $s_2$ in $\BR^+$, 
the norms $\norm{\cdot}{H_{s_1}^{1}}$ and 
$\norm{\cdot}{H_{s_2}^{1}}$ are equivalent 
{\cite{CMM1}}.  

\begin{notation}
\emph{We shall denote the space $H_s^{1}(M)$ simply by $\hu{M}$,
and we endow $\hu{M}$ with the norm $H_{1}^{1}(M)$.}
\end{notation}
\begin{definition}
The space $BMO(M)$ is the
space of all locally integrable functions~$f$ such that $N(f) < \infty$,
where
$$
N(f)
=  \sup_{B\in \cB_{1}} \frac{1}{\mu(B)}
\int_B \mod{f-f_B} \wrt\mu,
$$
and $f_B$ denotes the average of $f$ over $B$. 
We endow $BMO(M)$ with the ``norm''
$$
\norm{f}{BMO}
= N(f).
$$
\end{definition}

\begin{remark}
It is straightforward to check that $f$ is in $BMO(M)$
if and only if its sharp maximal function $f^{\sharp}$,
defined by 
$$
f^{\sharp}(x)
= \sup_{B \in \cB_{1}(x)} \frac{1}{\mu(B)}
\int_B \mod{f-f_B} \wrt\mu
\quant x \in M,
$$
is in $\ly{M}$.  Here $\cB_{1}(x)$ denotes the family
of all balls in $\cB_{1}$ that contain the point $x$. 
\end{remark}
The Banach dual of $\hu{M}$ 
is isomorphic to  $BMO(M)$ \cite[Thm~5.1]{CMM1}.
\par
\medskip
Now we recall the definition of the new Hardy spaces $\Xh{M}$,
introduced in \cite{MMV2}.
For every $\si$ in $\BR^+$ denote by $\cU_\si$ the operator 
$\cL\, (\si I + \cL)^{-1}$. 
Observe that
$$
\cU_\si
= \cI-\si \, (\si I + \cL)^{-1}.
$$
It is known that $\cU_\si$ is injective on $\lu{M}+\ld{M}$.  

\begin{definition} \label{def: Hardy space}
For each positive integer $k$ 
we denote by $\Xh{M}$ the Banach space of all
$\lu{M}$ functions~$f$ such that $\cU_{\be^2}^{-k}f$ is in $\hu{M}$, 
endowed with the norm
$$
\norm{f}{\Xk}
= \norm{\cU_{\be^2}^{-k} f}{H^1}.
$$
\end{definition}

\noindent
Clearly $\cU_{\be^2}^{-k}$ is
an isometric isomorphism between $\Xh{M}$ and $\hu{M}$.

\begin{definition} \label{def: BMO space}
For each positive integer $k$
we denote by $\Yh{M}$ the Banach dual of $\Xh{M}$.
\end{definition}

\begin{remark}
Since $\cU_{\be^2}^{-k}$ is
an isometric isomorphism between $\Xh{M}$ and $\hu{M}$,
its adjoint map~$\bigl(\cU_{\be^2}^{-k}\bigr)^*$ is an isometric isomorphism 
between the dual of $\hu{M}$, i.e., $BMO(M)$, and $\Yh{M}$.
Hence
$$
\bignorm{\bigl(\cU_{\be^2}^{-k}\bigr)^*f}{\Yk}
= \norm{f}{BMO}.
$$
\end{remark}

\begin{remark} \label{rem: fund prop Xh}
We recall the following properties of the spaces $\Xh{M}$,
proved in \cite{MMV2}:
\begin{enumerate}
\item[\itemno1]
if $\si$ is in $(\be^2-b, \infty)$, then $\cU_\si\!\!^k \hu{M}$ 
agrees with $\Xh{M}$;      
\item[\itemno2]  $\cU_{\beta^2}: X^{k-1}(M)\to \Xh{M}$ is an isomorphism for every positive integer $k$;
\item[\itemno3]
$
\hu{M} \supset X^1(M) \supset X^2(M) \supset \cdots
$
with proper continuous inclusions;
{\item[\itemno4]
if $1/p = 1-\te/2$, then the complex interpolation space 
$(\Xh{M},\ld{M})_{[\te]}$ is $\lp{M}$ (this is the analogue for
the spaces $\Xh{M}$ of the celebrated result of C.~Fefferman
and Stein \cite{FeS}).}

\end{enumerate}
\end{remark}

\section{Special atoms}
\label{s: Special atoms}

Atoms in $\Xh{M}$ will be $\ld{M}$ functions supported in a ball $B$
that satisfy a size condition analogous 
to that for $H^1$-atoms and an infinite dimensional cancellation 
condition, which will be expressed as orthogonality
to the space of ``$k$-quasi-harmonic'' functions
on $\OV{B}$ defined below.

\begin{definition}
Suppose that $k$ is a positive integer, and that
$B$ is a ball in $M$. We say that a function $V$ in $\ld{M}$ 
is \emph{$k$-quasi-harmonic} on $\OV{B}$ if 
$\cL^k V$ is constant (in the sense of distributions) 
in a neighbourhood of $\OV{B}$. We shall denote by  $\HBh$ the space of  
$k$-quasi-harmonic functions
on $\OV{B}$.
\end{definition}

\begin{remark} \label{rem: elliptic reg}
Observe that the following are equivalent:
\begin{enumerate}
\item[\itemno1]
$V$ is in $\HBh$;
\item[\itemno2]
$V$ is in $\ld{M}$ and is smooth in a neighbourhood of $\OV{B}$,
and $\cL^k V$ is constant therein. 
\end{enumerate}

Indeed, if $V$ is in $\HBh$, then $V$ is in $\ld{M}$ by the definition
of the space $\HBh$,  and $\cL^k V$ is a constant in the sense
of distributions in a neighbourhood of $\OV{B}$.
Hence $V$ is smooth on that neighbourhood by
elliptic regularity.

The converse is obvious.
\end{remark}

\noindent
Observe the following inclusions, which are direct consequences
of the definition of~$\HBh$:
$$
Q_B^1 \subset Q_B^2 \subset \cdots;
\qquad
(Q_B^1)^\perp
\supset (Q_B^2)^\perp
\supset \cdots
$$
\par\medskip
For each ball $B$ in $M$ we denote by $\ldO{B}$ 
the subspace of $\ld{M}$ consisting 
of all $\ld{M}$ functions $f$ with support contained in the
ball $\overline{B}$, and satisfying $\int_B f \wrt \mu = 0$.
\begin{proposition}  \label{p: canc II}
Suppose that $k$ is a positive integer, and that $B$ is a ball in $M$.  
The following hold:
\begin{enumerate}
\item[\itemno1]
$
\HBhperp
= \set{F\in L^2(M):\cL^{-k} F\in\ldO{B}};
$
\item[\itemno2]
$\cL^{-k} \bigl( \HBhperp\bigr)$ is contained in $\ldO{B} \cap 
\Dom(\cL^{k})$.  Furthermore, functions in $\HBhperp$
have support contained in $\OV{B}$;
\item[\itemno3]
$\Jbemenoh \bigl( \HBhperp\bigr)$ is contained in $\ldO{B}$.
\end{enumerate}
\end{proposition}

\begin{proof}
We prove \rmi.  First we show that $\HBhperp$
is contained in \break $\set{F\in L^2(M): \cL^{-k} F\in\ldO{B}}$.

Suppose that $F$ is in $\HBhperp$.  To show that the
support of $\cL^{-k}F$ is contained in $\OV{B}$
it suffices to prove that $(\cL^{-k}F, \One_{B'}) = 0$
for every ball $B'$ contained in $(\OV{B})^c$.  
Since $\cL$ is self adjoint,
$$
(\cL^{-k} F, \One_{B'})
= (F,\cL^{-k} \One_{B'}).
$$
Notice that $\cL^{-k} \One_{B'}$ is in $\HBh$, hence the last
inner product vanishes, as required.

Next we prove that the integral of $\cL^{-k}F$ is $0$.
Since the support of $\cL^{-k}F$ is contained in~$\OV{B}$
and $\cL$ is self adjoint,
$$
\int_M \cL^{-k}F \wrt \mu
= (\cL^{-k}F, \One_{2B})
= (F, \cL^{-k}\One_{2B}).
$$
Now, the last inner product vanishes, because $F$ is in $\HBhperp$ by assumption
and $\cL^{-k}\One_{2B}$ is in $\HBh$, as required. 

Next we prove that $\set{F\in L^2(M):\cL^{-k} F\in\ldO{B}}$
is contained in $\HBhperp$.
Suppose that  $\cL^{-k} F$ is in $\ldO{B}$.
Observe that $F$ is in $\Dom(\cL^k)$ and that
$F =  \cL^k \cL^{-k} F$.  Suppose now that $V$ is in $\HBh$.
Then $V$ is smooth in a neighbourhood of $\OV{B}$ by 
Remark~\ref{rem: elliptic reg}, and 
$$
\begin{aligned}
(F,V) 
= (\cL^k \cL^{-k} F, V)
= (\cL^{-k} F, \cL^k V) 
= 0.
\end{aligned}
$$
The last equality follows from the facts that $\cL^k V$ is constant
in a neighbourhood of $\OV{B}$, 
and that $\cL^{-k}F$ is in $\ldO{B}$, so that its integral on $B$ vanishes. 

Next we prove \rmii.
Clearly if $F$ is in $\ld{M}$, then $\cL^{-k}F$ is in $\Dom(\cL^k)$ 
by abstract set theory.  Moreover $\cL^{-k}F$ is in $\ldO{B}$ by \rmi,
and the first statement of \rmii\ follows.

To prove the second statement of \rmii, observe that the support of 
$\cL^{-k}F$ is contained in $\OV{B}$,
hence so is the support of $\cL^k\cL^{-k}F$, i.e., of $F$.

Finally, we prove \rmiii.  Observe that $\Jbemenoh = (\cI+\be^2\, \cL)^k \,
\cL^{-k}$.  Since  $\cL^{-k} \bigl( \HBhperp\bigr)$ 
is contained in $\ldO{B} \cap \Dom(\cL^{k})$ by \rmii, it suffices to show
that $\cL^j\bigl(\ldO{B} \cap \Dom(\cL^{k})\bigr)$ is contained in
$\ldO{B}$ for all $j$ in $\{0,1,\ldots, k\}$.  Suppose that $F$
is in $\ldO{B} \cap \Dom(\cL^{k})$.  Denote by $\phi$ a function
in $C_c^\infty(M)$ such that $\phi=1$ on $\OV{B}$.  Since $\cL$ is self adjoint
and the support of $F$ is contained in $\OV{B}$, 
$$
\int_M \cL^j F \wrt \mu
= (\cL^j F, \phi)
= (F, \cL^j \phi)
= 0,
$$
as required to conclude the proof of \rmiii\ and of the proposition. 
\end{proof}

\begin{definition}
Suppose that $k$ is a positive integer. 
An $X^k$-\emph{atom} associated to
the ball $B$ is a function $A$ in $\ld{M}$, supported in $B$, such that
\begin{enumerate}
\item[\itemno1]
$A $ is in $\HBhperp$;
\item[\itemno2]
$\ds\norm{A}{2}\leq \mu(B)^{-1/2}$.
\end{enumerate}
Note that condition \rmi\ implies that 
$\int_M A\wrt\mu=0$, because $\One_{2B}$ is in $\HBh$.
\end{definition}

\begin{remark} \label{rem: 12h atomi}
Note that if $A$ is a $X^k$-atom supported in $B$, then 
$\cL^{-k}A/\opnorm{\cL^{-k}}{2}$ is 
a $H^1$-atom with support contained in $\OV{B}$.

Indeed $A$ is in $\HBhperp$, so that $\cL^{-k}A$ is in $\ldO{B}$ by
Proposition~\ref{p: canc II}~\rmiii.  Moreover 
$$
\begin{aligned}
\norm{\cL^{-k}A}{2}
& \leq \opnorm{\cL^{-k}}{2} \, \norm{A}{2} \\
& \leq \opnorm{\cL^{-k}}{2} \, \mu(B)^{-1/2},
\end{aligned}
$$
so that $\cL^{-k}A/\opnorm{\cL^{-k}}{2}$ is a $H^1$-atom supported in $B$,
as required.   

Note also that an $X^k$-atom $A$ is in $\Xh{M}$ and
\begin{equation} \label{f: norma atomica}
\norm{A}{\Xk} 
\leq \opnorm{\Jbemenoh}{2}.
\end{equation}
Indeed, the function $\Jbemenoh A$ is in $\ldO{B}$
by Proposition~\ref{p: canc II}~\rmiii\ and 
$$
\begin{aligned}
\norm{\Jbemenoh A}{2}
& \leq \opnorm{\Jbemenoh}{2} \, \norm{A}{2} \\
& \leq \opnorm{\Jbemenoh}{2} \, \mu(B)^{-1/2}.
\end{aligned}
$$
Therefore $\Jbemenoh A/\opnorm{\Jbemenoh}{2}$ is an $H^1$-atom, and the required
estimate follows from the definition of $\Xh{M}$.
\end{remark}

\begin{definition}
Suppose that $k$ is a positive integer.  The space $\Xhat{M}$
is the space of all functions $F$ in $\hu{M}$ that admit a 
decomposition of the form $F= \sum_j \la_j\, A_j$, where $\{\la_j\}$ is 
a sequence in $\ell^1$ and $\{A_j\}$ is a sequence 
of $X^k$-atoms supported in balls $B_j$ in $\cB_1$.  Atoms supported
in balls in $\cB_1$ will be called \emph{admissible}.  We endow 
$\Xhat{M}$ with the norm 
$$
\norm{F}{\Xkat}
= \inf\, \bigl\{\sum_{j} \mod{\la_j}:  F = \sum_{j} \la_j \, A_j,\quad
\hbox{$A_j$ admissible $X^k$-atoms}\bigr\}.
$$
\end{definition}

\section{
The atomic decomposition of $\Xh{M}$.}
\label{s: The main result}
\
In this section we prove that $\Xh{M}=\Xhat{M}$ with equivalent norms. 
We need two lemmata.

\begin{lemma}\label{l: Usk}
If $\sigma>\beta^2-b$ the operator $\cU_\sigma^k$ is bounded on $H^1(M)$
for every {positive} integer~$k$.
\end{lemma}
\begin{proof} Denote by $\cD$ the operator $\sqrt{\cL-b}$ and by $m_{\sigma,k}$ the function
defined by
$$
m_{\sigma,k}(\zeta)=\left(\frac{\zeta^2+b}{\zeta^2+b+\sigma}\right)^k.
$$ 
Then $\cU_\sigma^k=
m_{\sigma,k}(\cD)$. 
The function $m_{\sigma,k}$ is bounded, even and holomorphic 
in the strip $\bS_{\beta}=\set{\zeta\in\BC: \mod{\Im \zeta}<\beta}$ and there 
exists a constant $C$ such that
\begin{equation}
\mod{D^j m_{\sigma,k}(\zeta)} 
\leq {C} \, {(1+\mod{\zeta})^{-j}}
\quant \zeta\in \bS_{\beta} \quant j \in \BN.
\end{equation}
The conclusion follows from \cite[Thm 3.4]{MMV2}.
\end{proof}

\noindent
The main step in the proof of the atomic decomposition of $\Xh{M}$ 
is Lemma~\ref{l: dec resolvent} below, 
which will be proved in Section~\ref{s: proof of Lemma}.
\begin{lemma} \label{l: dec resolvent}
Suppose that $k$ is a positive integer and that 
$M$ has $C^k$ bounded geometry. If
$A$ is an admissible $X^{k-1}$-atom  then $\Jbekappa A$ is 
in $\Xhat{M}$, and 
there exists a constant $C$, independent of $A$, such that
$$
\norm{\Jbekappa A}{\Xkat}
\leq C ,
$$
{where $\kappa$ is the constant which appears 
in the lower bound of the Ricci tensor.}
\end{lemma}

\noindent
The main result of this section is the following.

\begin{theorem} \label{t: atomic dec}
Suppose that $k$ is a positive integer and that 
$M$ has $C^k$ bounded geometry (see Definition~\ref{def: bounded
geometry}).
Then $\Xh{M}$ and $\Xhat{M}$
agree as vector spaces and there exists a constant $C$ such that
\begin{equation} \label{f: atomic dec}
C \, \norm{F}{\Xkat}
\leq \norm{F}{\Xk}
\leq \bigopnorm{\Jbe^{-k}}{2} \, \norm{F}{\Xkat}
\quant F \in \Xh{M}.
\end{equation}
\end{theorem}

\begin{proof}
First we prove that $\Xhat{M}$ is contained in $\Xh{M}$, and that 
the right-hand inequality in (\ref{f: atomic dec}) holds.  

Suppose that $F= \sum_{j} \la_j \, A_j$, where $A_j$ is an
admissible $X^k$-atom.  By Remark~\ref{rem: 12h atomi}
(see (\ref{f: norma atomica})), the function 
$\Jbemenoh A_j/\opnorm{\Jbemenoh}{2}$
is an $H^1$-atom.  
The series $\sum_j \la_j \, \Jbemenoh A_j$ is then convergent in 
$\hu{M}$.  Denote by $f$ its sum.   Since $\Jbeh$ is bounded
on $\hu{M}$ by Lemma \ref{l: Usk}, 
$$
\Jbeh f 
= \sum_j \la_j \, \Jbeh\bigl(\Jbemenoh A_j\bigr)
= F.
$$
Thus, $F$ is in $\Xh{M}$, and 
$$
\norm{F}{\Xk}
= \norm{f}{H^1}  
\leq \, \bigopnorm{\Jbe^{-k}}{2} 
\, \sum_j \mod{\la_j}.
$$ 
The right-hand inequality in (\ref{f: atomic dec}) follows
from this by taking the infimum over all the decompositions
of $F$ of the form $F = \sum_j \la_j \, A_j$.

Next we prove that $\Xh{M}$ is contained in $\Xhat{M}$ and that 
the left-hand inequality in (\ref{f: atomic dec}) holds.  
For notational convenience, in the rest of this proof
we denote $\hu{M}$ also by $X^0(M)$, and write $\cR$ instead of $\Rbekappa$
and $\cU$ instead of $\Jbekappa$.  

We argue inductively.  
The result is trivial in the case where $k=0$, because $\cU^0= \cI$.
Suppose that the result holds for $k-1$ and that 
 $F$ is in $\Xh{M}$.  Then $f=\cU^{-1} F$
is a function in $\Xhmu{M}$, and 
$\norm{f}{X^{k-1}} = \norm{F}{\Xk}$, by Remark \ref{rem: fund prop Xh}.

By the induction hypothesis for each $\vep$ in $\BR^+$ 
there exist a sequence $\{A_j\}$ of 
admissible $X^{k-1}$-atoms and a summable 
sequence~$\{c_j\}$ of complex numbers such that 
\begin{equation}
f = \sum_{j} c_j \, A_j
\qquad\hbox{and}\qquad
\norm{f}{X^{k-1}} 
\geq \sum_j \mod{c_j} - \vep.
\end{equation}
Observe that we may write
\begin{equation} \label{f: series for F}
F
=\cU f 
= \sum_j c_j \, \ \cU  A_j ,
\end{equation}
because
the series $\sum_j c_j \, A_j$ converges to $f$ in $\hu{M}$,
and $\cU$ is bounded on~$\hu{M}$ by Lemma \ref{l: Usk}.

Fr{}om (\ref{f: series for F}) and Lemma~\ref{l: dec resolvent}, 
we see that 
$$
\begin{aligned}
\norm{F}{\Xkat}
& \leq  \sum_j \mod{c_j} \,  \norm{\cU A_j}{\Xkat} \\
& \leq C \, 
        \sum_j \mod{c_j}  \\
& \leq C \, 
       \bigl( \norm{f}{X^{k-1}} + \vep\bigr) \\
& =    C \, 
       \bigl( \norm{F}{\Xk} + \vep\bigr).
\end{aligned}
$$
Therefore $F$ is in $\Xhat{M}$, and $\norm{F}{\Xkat}
\leq C \, \norm{F}{\Xk}$, as required.  

This concludes the proof of the theorem.
\end{proof}

\begin{remark}
Suppose that $k$ is a positive integer
and that $s$ is a scale parameter in $\BR^+$.  
The space of all functions $F$ in $\hu{M}$ that admit a 
decomposition of the form $F= \sum_j \la_j\, A_j$, where $\{\la_j\}$ is 
a sequence in $\ell^1$ and $\{A_j\}$ is a sequence 
of $X^k$-atoms supported in balls $B_j$ in $\cB_s$ agrees with 
$\Xhat{M}$ (hence with $\Xh{M}$).
The norm on $\Xhat{M}$ defined by
$$
\inf\, \bigl\{\sum_{j} \mod{\la_j}:  F = \sum_{j} \la_j \, A_j,\quad
\hbox{$A_j$  are $X^k$-atoms supported in  balls 
of $\cB_s$}\bigr\}
$$
is an equivalent norm on $\Xhat{M}$.

To prove this, it suffices to observe that
minor modifications in the proof of Theorem~\ref{t: atomic dec}
and Lemma~\ref{l: dec resolvent} show that $F$ is in $\Xh{M}$
if and only if it admits a decomposition in terms
of $X^k$-atoms supported in balls in $\cB_s$. 
\end{remark}

\begin{remark}
Suppose that $p$ is in $(1,\infty)$
and denote by $p'$ the index conjugate to $p$. 
Assume that $k$ is a positive integer and that $B$ is in $\cB$.
Define ${Q}_{B,p'}^k$ to be the space of all functions $V$
in $L^{p'}(M)$ such that $\cL^k V$ is constant (in the sense
of distributions) in a neighbourhood of $\OV{B}$.
Then denote by $({Q}_{B,p'}^k)^\perp$ the annihilator of 
${Q}_{B,p'}^k$ in $\lp{M}$. Then a $X^k$-atom in $\lp{M}$ is
an element $A$ of $({Q}_{B,p'}^k)^\perp$, satisfying the size condition
$$
\norm{A}{p} 
\leq \mu(B)^{-1/p'}.
$$
It is straightforward to modify
the theory of  this section to show 
that $\Xh{M}$ admits an atomic characterisation
in terms of $X^k$-atoms in $\lp{M}$.  The fact that $\cU$
is an isomorphism of $\lp{M}$ for all $p$ in $(1,\infty)$
plays an important r\^ole here. 
\end{remark}

As a consequence of the atomic decomposition
of the space $\Xh{M}$, we may describe explicitly the
action of elements of $Y^k(M)$, the dual of $\Xh{M}$, on 
finite linear combinations of $X^k$-atoms.

\begin{definition}
Suppose that $k$ is a positive integer.  We denote by 
$\Xhfin{M}$ the vector space of all finite linear
combinations of $X^k$-atoms and by $\hufin{M}$ 
the vector space of all finite linear
combinations of $H^1$-atoms.
\end{definition}

Suppose that $\ell$ is a continuous linear functional on $\Xh{M}$.
Since $\Jbeh$ is an isomorphism between $\hu{M}$ and $\Xh{M}$
and $\Xh{M}$ and $\Xhat{M}$ are isomorphic by Theorem~\ref{t: atomic dec},
$\ell\circ \Jbeh$ is a continuous linear
functional on $\hu{M}$.  By \cite[Thm~5.1]{CMM1}, there exists a function
$f$ in $BMO(M)$ such that 
$$
(\ell\circ\Jbeh) (g)
= \int_M g \, f \wrt \mu
\quant g \in \hufin{M}.
$$
Clearly 
$$
\norm{\ell}{Y^k} 
= \norm{\ell\circ\Jbeh}{(H^1)^*} 
= \norm{f}{BMO}.
$$

It may be worth describing how the functional $\ell$
acts on $X^k$-atoms, or, more generally, on functions
in $\Xhfin{M}$.
Suppose that $A$ is a $X^k$-atom with support contained in an
arbitrary ball $B$.  
Since $\Jbe^{-k} = \bigl(\cI+\be^2\cL^{-1}\bigr)^k$,
there exist constants $c_j$ such that
$$
\Jbe^{-k}A
= \sum_{j=0}^k c_j \, \cL^{-j}A.
$$
Then $\Jbe^{-k}A$ is a finite linear combination of
$H^1$-atoms by Remark~\ref{rem: 12h atomi}.  
Therefore
$$
\begin{aligned}
\ell(A) 
&  = (\ell\circ\Jbeh) (\Jbe^{-k} A) \\
& = \int_M (\Jbe^{-k}A) \, f \wrt \mu.
\end{aligned}
$$
Observe that $\Jbe^{-k}A$ is supported in $B$, so that the last 
integral is just the inner product in $\ld{M}$
between $\Jbe^{-k}A$ and $\One_B\, f$.  Since $\Jbe^{-k}$
is self adjoint, we may write
$$
\ell(A) 
= \int_M A \, \, \Jbe^{-k}(\One_B f) \wrt \mu.
$$
A similar argument shows that
if $F$ is in $\Xhfin{M}$ and its support is contained
in the ball $B$, then 
$$
\ell(F)
= \int_M F \, \, \Jbe^{-k}(\One_B f) \wrt \mu.
$$

It may be worth observing that a consequence of
this representation formula for $\ell$, and of 
the fact that $\norm{\ell}{Y^k} = \norm{f}{BMO}$,
is that 
$$
N'(f) \leq \norm{\ell}{Y^k} \leq \opnorm{\Jbeh}{2} \, N'(f),
$$
where
$$
N'(f) 
= \sup_{B\in \cB_1} \Bigl(\frac{1}{\mu(B)} \int_B
\bigmod{\Jbe^{-k} (\One_B\, f) 
- f_B \,\Jbe^{-k} \One_B }^2 \wrt\mu \Bigr)^{1/2}
$$
and $f_B$ denotes the average of $f$ over $B$.
The proof of this fact is straightforward and is omitted.

\section{Riesz transforms of even order}\label{s: Riesz}
Denote by $\nabla$ the covariant derivative on $M$. The Riesz 
transform of order $\ell$ is the operator $\nabla^{\ell}\cL^{-\ell/2}$ 
mapping smooth functions 
with compact support on $M$ to sections of the bundle $T_\ell(M)$ of 
covariant tensors of order $\ell$. In this section we exploit the atomic decomposition 
of the spaces $X^k(M)$ to prove that the Riesz
transforms of even order $\nabla^{2k}\cL^{-k}$ extend to bounded operators
from $X^k(M)$ to the space $L^1\big(T_{2k}(M)\big)$ of $L^1$ sections of $T_{2k}(M)$.
To prove this result we need to strengthen the bounded 
geometry assumption on $M$, by replacing the derivatives of the Ricci 
tensor with those 
of the Riemann tensor in Definition \ref{def: bounded geometry}.
\begin{definition}\label{sbg}
We say that $M$ has $C^\ell$ \emph{strongly bounded} geometry
if the injectivity radius is positive and the following hold:
\begin{itemize}
\item{} 
if $\ell =0$, then the Ricci tensor is bounded from below
$\Ric \geq -\kappa^2$ for some positive $\kappa$;
\item{} 
if $\ell$ is positive, then the covariant
derivatives $\nabla^j R$ of the Riemann tensor are uniformly
bounded on $M$ for all $j$ in $\{0,\ldots, \ell\}$.
\end{itemize}
\end{definition} 

\noindent
We recall that the boundedness of the first order Riesz transform on $L^2(M)$ 
follows from the identity $\cL=\nabla^*\nabla$ and the self-adjointness 
of $\cL$ on $L^2(M)$ \cite{Str}. 
From a result of Aubin \cite[Prop. 3]{Au}, it follows 
also that  if $b>0$ and $M$ has $C^{\ell-2}$ strongly bounded 
geometry then  the Riesz 
transform 
of order $\ell\ge2$ extends to a bounded operator from $L^2(M)$  to 
the space $L^2\big(T_\ell(M)\big)$ of square integrable sections of $T_\ell
(M)$.

In \cite{MMV2} the authors,
{under the additional assumption that $b=\be^2$}, proved that 
the Riesz transform of order $1$ 
maps {$\Xh{M}$} to $L^1\big(T_1(M)\big)$ {for $k$ large enough}.  
{In general}, the Riesz transforms of 
{order one} do not map $H^1(M)$ to $L^1\big(T_\ell(M)\big)$. 
A counterexample on noncompact symmetric spaces will appear in a 
forthcoming paper of the authors \cite{MMV3}.
{Notice that the modified Riesz transform of order $1$, i.e.
the operator $\nabla (\cL+\vep I)^{-1/2}$, {for $\vep >0$}, maps $\hu{M}$ into $\lu{T_1(M)}$
even when $M$ satisfies less stringent assumptions on $M$ \cite{Ru}}.

\begin{theorem}\label{t: RT}
If $b>0$ and $M$ has $C^{2k-2}$ strongly bounded geometry then 
the Riesz transform of order  $2k$ 
extends to a bounded operator from $X^k(M)$ to {$\lu{T_{2k}(M)}$ and from
$\lp{M}$ to $\lp{T_{2k}(M)}$ for all $p$ in $(1,2)$}.
\end{theorem}
\begin{proof}
{To prove that $\nabla^{2k}\cL^{-k}$ is bounded 
from $X^k(M)$ to $\lu{T_{2k}(M)}$}
it suffices to show that $\nabla^{2k}\cL^{-k}$ maps $X^k$-atoms 
into $L^1\big(T_{2k}(M)\big)$ uniformly. Suppose that $A$ is a
 $X^k$-atom associated to the ball $B$. Then, by Remark \ref{rem: 12h atomi}, 
 the function $\cL^{-k} A$ is supported in $\overline{B}$. 
 Hence, by H\"older's inequality and the boundedness of 
$\nabla^{2k}\cL^{-k}$ on $L^2(M)$,
\begin{align*}
\norm{\nabla^{2k}\cL^{-k} A}{1}&\le \norm{\nabla^{2k}\cL^{-k} A}{2}
\, \mu(B)^{1/2}\\ 
&\le C\, \norm{ A}{2} \ 
\mu(B)^{1/2}\le \,C,
\end{align*} 
{as required.  The boundedness of $\nabla^{2k}\cL^{-k}$ from 
$\lp{M}$ to $\lp{T_{2k}(M)}$ for all $p$ in $(1,2)$ 
follows by interpolation
(see \cite[Thm~2.15]{MMV2})
from {the $X^k(M)--\lu{T_{2k}(M)}$ boundedness} and the aforementioned 
result of Aubin.}
\end{proof}

\section{Proof of Lemma \ref{l: dec resolvent} }
\label{s: proof of Lemma}
In this section we shall prove Lemma \ref{l: dec resolvent}. 
First we need a variant of the ``economical decomposition of atoms" proved in \cite[Lemma 5.7]{MMV2}.

\begin{lemma} \label{l: economical decomposition I}
Suppose that $k$ is a positive integer and that 
$M$ has $C^k$ bounded geometry (see Definition~\ref{def: bounded geometry}).
If  $a$ is an $H^1$-atom in $\Dom(\cL^k)$,
then $\cL^k a$ is in $\Xhat{M}$.
Furthermore, if the support of $a$ is contained  in the ball $B$, 
then there exists a constant $C$ such that
$$
\norm{\cL^k a}{\Xkat} 
\leq C\, (1+ \,r_B)\, \mu(B)^{1/2}\, \norm{\cL^k a}{2}.
$$
\end{lemma}

\begin{proof} 
Suppose  first that the support of $a$ is contained in 
a ball $B$ such that $r_B\le 1$. 
Since $\cL^ka$ is in $\HBhperp$ by Proposition \ref{p: canc II}, 
$\mu (B)^{-1/2}{\cL^k a}/\norm{\cL^k a}{2}$ is a $X^k$-atom supported 
in a ball in $\cB_1$ and the lemma is proved. \par
Next, suppose that $r_B>1$. 
Denote by $\fS$ a $1/3$-discretisation 
of $M$, i.e. a set of points  in $M$ that is
maximal with respect to the property
$$
\min\{d(z,w): z,w \in \fS, z \neq w \} >1/3,
\quad\hbox{and}\quad d(\fS, x) \leq 1/3 \quant x \in M.  
$$
The family $\set{B(z,1):z\in \fS}$ is a covering of $M$ which is 
uniformly locally finite,  by the uniform ball size and the {local 
doubling} property of the Riemannian measure 
(see, for instance, \cite[Theorem~3.10]{Ch}). 
By the same token, the set $B\cap\fS$ is  
finite and has at most $N$ points $z_1,\ldots,z_N$, with $N \le 
C\,\mu(B)$, where $C$ is a constant which does not depend on $B$. 
Denote  by $B_j$  the ball with centre $z_j$ and radius $1$,
and by $\set{\psi_j: j=1,\ldots,N}$ a partition of unity on $B$ 
subordinated to the covering $\set{B_j:j=1,\ldots,N}$.

Fix~$j$ in $\{1,\ldots,N\}$ and denote by $z_j^0,\ldots, z_j^{N_j}$
points on a minimizing geodesic joining $z_j$ and $c_B$, 
with the property that 
$z_j^0 = z_j$, $z_j^{N_j} = c_B$, and $d(z_j^h,z_j^{h+1})$ is 
approximately equal to $1/3$. Note that $N_j\le 4r_B$.  
Denote by $B^h_j$ the ball $B(z^h_j,1/12)$, 
for $j=1,\ldots,N$ and $h=0,\ldots, N_j$. Then the balls 
$B^h_j$ are disjoint, $B^h_j\subset B(z_j^h,1)\cap B(z_j^{h+1},1)$ 
and $B_j^{N_j}=B(c_B,1/12)$. \par
Denote by $\phi^h_j$ a nonnegative function in $C^\infty_c(B^h_j)$ that
has integral $1$. By the uniform ball size property we may choose the 
functions $\phi^h_j$ so that there exists a constant $A$ 
such that $\norm{\phi^h_j}{2}\le A$ for all $h$ and $j$.

The existence of a uniform bound on the derivatives of the 
Ricci tensor implies that 
we can choose the functions $\psi_j$ and $\phi^h_j$ so that their covariant 
derivatives of order up to $2k$ are uniformly bounded for all $j$ and $h$ 
(see \cite[p. 14]{He}). 
Now, denote by $a_j^0$ the function $a\, \psi_j$.  Clearly
$$
a 
= \sum_{j=1}^N \psi_j \, a
= \sum_{j=1}^N  \, a_j^0. 
$$
Next, define
$$
a_j^1
= a_j^0-\phi_j^0 \, \int_M a_j^0\wrt\mu
\quad\hbox{and}\quad
a_j^h=(\phi_j^{h-2}-\phi_j^{h-1})\int_M a_j^0\wrt\mu,
\quad 2\le h\le N_j+1.
$$
Then, for every $h$ in $\set{1,\ldots,N_j}$,
the support of $a_j^h$ is contained in $B(z_j^{h-1},1)$, the integral 
of  $a_j^h$ vanishes and  
$$
\begin{aligned}
\norm{a_j^h}{2}
& \le 2A \int_M \mod{a_j^0}\wrt\mu \\
& \le C\, \norm{a_j^0}{2} \, \mu(B_j)^{1/2}\\
& \le C\, \norm{a_j^0}{2} \, \mu(B_j^h)^{-1/2}.
\end{aligned}
$$ 
In the last two inequalities we have used 
the uniform ball size property (\ref{f: ubsc}).   
Hence there exists a constant $C$, independent of $j$ and $h$,
such that 
\begin{equation} \label{f: normaHuno ajk}
\norm{a_j^h}{H^1}
\leq C \, \norm{a_j^0}{2}.
\end{equation}
Moreover
$$
a_j^0=\sum_{h=1}^{N_j+1}a_j^h+\phi_j^{N_j}\int_M a_j^0\wrt\mu.
$$
Thus
$$
a
= \sum_{j=1}^N\sum_{h=1}^{N_j+1} a^h_j,
$$
because $\sum_j \int_M a^0_j\wrt\mu=\int_M a\wrt\mu=0$ and all the functions 
$\phi_j^{N_j}$, $j=1,\ldots,N_j$ coincide, since $B^{N_j}_j = B(c_B,1/12)$.
Moreover, $a^1_j$ is in $\Dom(\cL^k)$, and
$$
\begin{aligned}
\norm{\cL^k a^1_j}{2}
& \le \norm{\cL^k a^0_j}{2}+\norm{\cL^k \phi^0_j}{2}
       \int_M \mod{a^0_j}  \wrt\mu  \\
& \le C\,\norm{\cL^k a^0_j}{2},
\end{aligned}
$$
where, in the last inequality, we have used the estimate
$\norm{a^0_j}{2}
\le C\, \norm{\cL^k a^0_j}{2}$, which holds because $\cL$ has spectral gap. 
Similarly, if $h=2,\ldots,N_j+1$, then  
$a^h_j$ is in $\Dom(\cL^k)$, and
$$
\norm{\cL^k a^h_j}{2}\le C\, \norm{\cL^k a^0_j}{2}.
$$
Hence $\cL^k a^h_j/\norm{\cL^k a^0_j}{2}$ is a multiple of 
a $X^k$-atom supported in a ball of radius $1$, with a constant $C$ 
which does not depend on $j$ and $h$  by the uniform ball size property.
Thus
\begin{equation}\label{ajk}
\norm{\cL^ka^h_j}{\Xkat}
\le C\, \norm{\cL^ka^0_j}{2} 
\quant j,h.
\end{equation}
Adding up the inequalities in (\ref{ajk}), we obtain
\begin{align*}
\norm{\cL^k a}{\Xkat}
&\le \sum_{j=1}^N \sum_{h=1}^{N_j+1} \norm{\cL^k a^h_j}{\Xkat}  \\ 
&\le C\,\sum_{j=1}^N\sum_{h=1}^{N_j+1} \norm{\cL^k a^0_j}{2}.
\end{align*}
Remembering   that $N_j\le C\,r_B$ and $N\le C\,\mu(B)$, and using Schwarz's
inequality, we see that the right-hand side is dominated by
\begin{align*}
C\, r_B \sum_{j=1}^N \norm{\cL^k a^0_j}{2}
&  \le C\, r_B\,N^{1/2} \, \Bigl(\sum_{j=1}^N
     \norm{\cL^k a^0_j}{2}^2\Bigr)^{1/2}\\
&  \le C\, r_B\,\mu(B)^{1/2} \, \norm{\cL^k a}{2}.
\end{align*}
In the last inequality we have used the fact that $\set{\psi_j}$ 
is a partition of unity on $B$, subordinated to the 
uniformly locally finite covering $\set{B_j}$.\par
This completes the proof of the lemma.
\end{proof}

\begin{remark}\label{c: large atoms} 
There exists a constant $C$ such that 
$$
\norm{f}{\Xkat}
\le C\, (1+r_B)\,\mu(B)^{1/2} \, \norm{f}{2} 
\quant f\in \HBhperp.
$$
Indeed, 
if $f$ is in $\HBhperp$, then the function  $\cL^{-k}f$ is a 
multiple of a $H^1$-atom, by Proposition~\ref{p: canc II}. 
The conclusion follows, by Lemma~\ref{l: economical decomposition I}.
\end{remark}

\noindent
The second ingredient in our proof of 
Lemma~\ref{l: dec resolvent} are two  technical results in
one-dimensional Fourier analysis (see Lemma~\ref{l: P} and 
Lemma~\ref{l: stima Runo} below). To state them we need some more notation.
For every $f$ in $\lu{\BR}$ 
define its Fourier transform $\wh f$~by
$$
\wh f(t) = \ir f(s) \, \e^{-ist} \wrt s
\quant t \in \BR.
$$
If $f$ is a function on $\BR$, and $\la$ is in $\BR^+$,
we denote by $f^\la$ and $f_\la$ the~$\la$-dilates of $f$, defined by
\begin{equation} \label{f: dilate} 
f^\la(x)
= f(\la x)  
\qquad\hbox{and}\qquad
f_\la(x)
= \la^{-1} \, f(x/\la)  \quant x \in \BR.
\end{equation}

\noindent
For each $\nu\geq -1/2$, denote by $\cJ_\nu: \BR\setminus\{0\} \to \BC$
the modified Bessel function of order $\nu$, defined by
$$
\cJ_\nu(t) = \frac{J_\nu(t)}{t^\nu},
$$
where $J_\nu$ denotes the standard Bessel function of the first
kind and order $\nu$ (see, for instance, 
\cite[formula~(5.10.2), p.~114]{L} for the definition).  Recall that 
$$
\cJ_{-1/2} (t)
= \sqrt{\frac{2}{\pi}}\  \cos t
\qquad\hbox{and that}\qquad
\cJ_{1/2} (t)
= \sqrt{\frac{2}{\pi}}\ \frac{\sin t}{t}. 
$$
For each positive integer $\ell$,
we denote  by $\cO^\ell$ the differential operator 
$t^\ell\, D^\ell$ on the real~line.  For the proof
of the following lemma, see \cite[Lemma~4.1]{MMV2}.

\begin{lemma}\label{l: P}
For every positive integer $k$ there exists a polynomial 
$P_{k+1}$ of degree $k+1$ without constant term, such that 
\begin{equation} \label{f: propertiesRiesz I} 
\ir {f} (t) \, \cos (vt) \wrt t
=   \ir  P_{k+1}(\cO)f(t)\,  \cJ_{k+1/2} (t v)  \wrt t,
\end{equation}
for all functions $f$  such that 
$\cO^\ell f\in \lu{\BR}\cap C_0(\BR)$ for all
$\ell$ in $\{0,1,\ldots, k+1\}$.
\end{lemma} 

Denote by $\om$ an even function in 
$C_c^\infty(\BR)$ which is supported in $[-3/4,3/4]$, is equal to~1
in $[-1/4,1/4]$, and satisfies 
$$
\sum_{j\in \BZ} \om(t-j) = 1
\quant t \in \BR.
$$
Denote by $\phi$ the function $\om^{1/4}- \om$, 
where $\om^{1/4}$ denotes the $1/4$-dilate of $\om$.  
Then $\phi$ is smooth, even and vanishes in the complement of the set
$\{t \in \BR: 1/4\leq \mod{t} \leq 4\}$.
For a fixed $R$ in $(0,1]$ and for each positive integer $i$, 
denote by $E_{i}$ the set
$\{t \in \BR: 4^{i-1}R \leq \mod{t} \leq 4^{i+1}R\}$.
Clearly $\phi^{1/(4^{i}R)}$ is supported in $E_{i}$,
and $\sum_{i=1}^\infty \phi^{1/(4^{i}R)} =~1$ in $\BR\setminus (-R,R)$.
Denote by $d$ the integer $[\!\![\log_4(3/R)]\!\!]+1$. 
To avoid cumbersome notation, we write $\rho_i$ instead of $1/(4^{i}R)$.
Then
\begin{equation} \label{f: dec om phi}
\om^{\rho_0} +  \sum_{i=1}^{d} \phi^{\rho_i} =~1 
\qquad\hbox{on}\quad [-3,3].
\end{equation}
Suppose that $c$ is in $\BR^+$, and
denote by $r$ the function defined by 
\begin{equation}\label{funzr}
r(\la) 
= \frac{1}{c^2+\la^2}
\quant \la \in \BC \setminus \{\pm ic\}.
\end{equation}
Note the decomposition
\begin{equation} \label{f: dec om hat r}
\wh{\om}* r(\lambda)= \sum_{i=0}^{d} S_i (\la),
\end{equation}
where the functions $S_i: \BR\to\BC$ are defined by
\begin{equation} \label{f: A0}
S_0(\la)
= \frac{1}{2\pi} \ir \om^{\rho_0}(t) \,\,  P_{N}(\cO)(\om\, \wh r)(t)
\,\,  \cJ_{N-1/2}(\la\, t) \wrt t 
\quant \la \in \BR,
\end{equation}
and, for $i$ in $\{1,\ldots,d\}$,
\begin{equation}  \label{f: Ai}
S_i(\la)
= \frac{1}{2\pi} \, \ir \phi^{\rho_i}(t)\, 
    P_{N}(\cO) (\om\, \wh r)(t) \, \cJ_{N-1/2}(\la t) \wrt t
\quant \la \in \BR,
\end{equation}
{where $N$ is a positive integer.}
\begin{remark} \label{rem: sigmak}
Note that there exist constants $c_\ell$ such that 
$$
t^{-1}P_{N}(\cO)
= \sum_{\ell=0}^{N-1} c_\ell \, t^\ell \, D^{\ell+1}.
$$
\end{remark}

\begin{lemma} \label{l: stima Runo}
Suppose that $N$ is a positive integer.
The following hold:
\begin{enumerate}
\item[\itemno1]
the norm
$\norm{t^{-1}P_{N}(\cO)
 \wh{r}}{\infty}$ is finite;
\item[\itemno2]
if $N\geq 3$, then there exists a constant $C$, 
independent of $R$ in $(0,1]$, such that 
$$
\sup_{\la\ge 0}\, (\la^2+1)\, \mod{S_0(\la)}
\leq C .
$$
\end{enumerate}
\end{lemma}
\begin{proof}
By Remark ~\ref{rem: sigmak},
 to prove \rmi\ it suffices to show that
\begin{equation}\label{tDr}
\sup_{t \in \BR} \mod{t^\ell \, D^{\ell+1} \wh{r}(t)}
< \infty
\quant \ell \in \{0,\ldots, N-1\}.
\end{equation}
This is a standard estimate in Fourier analysis.  
Recall that $\wh{r}(t) = (1/c)\, \e^{-c\mod{t}}$.  
It is straightforward to check that $D\wh r = -c\, \wh r\,\, \sgn$,
and that for each $k\geq 1$
$$
\begin{aligned}
D^{2k} \wh{r} 
=&\  c^{2k} \, \wh r - 2 \, \sum_{j=0}^{k-1} c^{2(k-1-j)} \, D^{2j} \de_0
\\
D^{2k+1} \wh{r} 
=&\, - c^{2k+1} \,  \wh r \cdot \sgn - 2 \, 
     \sum_{j=0}^{k-1} c^{2(k-1-j)}\, D^{2j+1} \de_0.
     \end{aligned}
$$
Hence
$$
t^{2k-1}\, D^{2k} \wh{r}(t) 
= c^{2k}\, t^{2k-1}\, \wh r(t)
\quad\hbox{and}
\quad
t^{2k}\, D^{2k+1} \wh{r} (t)
= - c^{2k+1}\,  t^{2k} \, \sgn(t)\,  \cdot \wh r(t),
$$
so that 
$$
\mod{t^\ell\, D^{\ell+1} \wh{r}(t)}
= c^{\ell} \, \mod{t}^\ell\, \e^{-c\mod{t}}
\quant t \in \BR,
$$
and the required estimate follows.

To prove \rmii, observe that, on the one hand, by (\ref{f: A0}) and (\ref{tDr})
$$
\begin{aligned}
\mod{S_0(\la)}
& \leq \norm{\om}{\infty} \, \norm{
t^{-1}P_{N}(\cO)
(\om\, \wh r)}{\infty}
    \ir \mod{t}\, \mod{\cJ_{N-1/2} (t \la)} \wrt t \\ 
& \leq  \, C\,\norm{\om}{\infty} \,
  \,\, \la^{-2} 
\quant \la \in [0,\infty). 
\end{aligned}
$$
On the other hand, the function $\cJ_{N-1/2}$ is bounded, so that
$$
\begin{aligned}
\mod{S_0(\la)}
& \leq C\, \norm{t^{-1}P_{N}(\cO)
 (\om\, \wh r)}{\infty} 
    \ir \om^{\rho_0}(t) \,\mod{t} \wrt t \\ 
& \leq C 
\quant \la \in [0,\infty). 
\end{aligned}
$$
We have used the fact that $\rho_0 = 1/R$ and $R\leq 1$ in the
last inequality.
The proof of the lemma is complete. 
\end{proof}

The third, and last, ingredient in the proof of Lemma \ref{l: dec resolvent} is the following proposition, which shows that certain functions of the operator $\cL$ map $H^1$-atoms
into functions that have integral $0$.
For technical reasons, it is convenient to work with functions of the wave propagator 
$$
\cD_1=\sqrt{\cL-b+\kappa^2}
$$
instead of functions of $\cL$. 
We recall that $-\kappa^2$ is the greatest lower 
bound of the Ricci curvature (see Basic assumptions~\ref{Ba: on M}). 
The  reason for considering the operator $\cD_1$ instead of  
$\cD=\sqrt{\cL-b}$ is that, in order to prove estimates of 
the gradient of the kernels associated to functions of $\cL$,  
we need to exploit the identity $\wrt\cL=\mathbb{L} \wrt$, 
where $\mathbb{L}$ is the Hodge Laplacian $\mathbb{L}$ on $1$-forms 
and $\wrt$ denotes exterior differentiation (see \cite[Prop. 5.5]{MMV1}). 
Whereas, in general,  the operator $\mathbb{L}-b$ is not a 
positive operator on $1$-forms, the operator 
$\mathbb{L}-b+\kappa^2$ is nonnegative on manifolds whose 
Ricci curvature satisfies the lower bound $\Ric\ge -\kappa^2$ .

\begin{proposition}  \label{p: Mean}
Suppose that $\nu$ is in $[-1/2,\infty)$, that $w$ is 
a complex measure on $\BR$ 
and that $a$ is an $H^1$-atom.  Define the operator $\cW_\nu(\cD_1)$
on $\ld{M}$ spectrally by
$$
\cW_{\nu}(\cD_1) f 
= \ir 
 \, \cJ_{\nu} (t\cD_1) f 
 \wrt w(t)
\quant f \in \ld{M}.
$$
The following hold:
\begin{enumerate}
\item[\itemno1]
$\int_M \cW_\nu(\cD_1) a \wrt \mu = 0$;
\item[\itemno2]
$\int_M S_i(\cD_1) a \wrt \mu = 0$ for $i=0,\ldots,d$ 
(the functions $S_i$ are defined in (\ref{f: A0}) and (\ref{f: Ai}).
\end{enumerate}
\end{proposition}

\begin{proof} 
A simple argument, based on the finite speed of propagation property 
of the operator  $\cL-b+\kappa^2$, shows that 
\begin{equation}\label{}
\int_M\cJ_\nu(t\cD_1)\,a\wrt\mu =0.
\end{equation}
(see \cite[Prop. 5.5]{MMV2}). Since $\cW_\nu$ 
and the function $S_i$ are integrals of 
$\cJ_\nu(t\cdot)$ with respect to complex measures, 
we obtain the desired conclusion by interchanging the order of integration.
\end{proof}

\begin{remark}\label{r: fps}
Note that for every $\nu$ in $[-1/2,\infty)$
the function $\la\mapsto\cJ_\nu(t\la)$ is even and of
entire of exponential type $t$, so that 
kernel $k_{\cJ_\nu (t\cD_1)}$ of the operator 
$\cJ_\nu(t\cD_1)$ 
is supported in 
the set $\{(x,y) \in M\times M:
d(x,y) \leq t\}$ by the finite propagation speed.
\end{remark}

\noindent
The main step in the proof of our main result
is Lemma~\ref{l: dec resolvent}, which we restate for the 
reader's convenience.  {The idea, used in the proof, of subordinating 
spectral functions of $\cL$ to the wave propagator 
has been used several times
since its appearance in \cite{CGT,Ta}}.

\begin{lemma*}[{\bf 4.2}]
Suppose that $k$ is a positive integer and that 
$M$ has $C^k$ bounded geometry.
Let
$A$ be an admissible $X^{k-1}$-atom.  Then $\Jbekappa A$ is 
in $\Xhat{M}$, and 
there exists a constant $C$, independent of $A$, such that
$$ 
\norm{\Jbekappa A}{\Xkat}
\leq C.
$$
\end{lemma*}

\begin{proof}  
Suppose that the atom $A$ is supported in the ball 
$B(p,R)$. Then $R\leq 1$, because $A$ is admissible.
Denote by $N$ an integer $>n/2+3$.
For notational convenience, in this proof 
we shall write $\cJ$ instead of $\cJ_{N-1/2}$,
$\cR$ instead of $\Rbekappa$, $\cU$ instead of $\Jbekappa$ 
and $c$ instead of $\sqrt{4\be^2+b }$.
Observe that $\cR = r(\cD_1)$ (the function $r$ was defined in (\ref{funzr})).
\par
\emph{Step I: splitting of the operator}.
Define the operators~$\cS$ and $\cT$ spectrally~by 
\begin{equation} \label{f: decomposition R}
\cS
= (\wh\om\ast r) (\cD_1)
\qquad\hbox{and}\qquad
\cT 
= (r-\wh\om\ast r) (\cD_1).
\end{equation}
Then
$\cU A=\cL
\cR A
=\cL \cS A+\cL \cT A.
$
We shall prove that both $\cL\cS A$ and $\cL\cT  A$ 
are in $\Xhat{M}$ and that there exists a constant $C$, independent
of $A$, such that 
\begin{equation} \label{f: required atomic}
\norm{\cL\cS A}{\Xkat}
\leq C
\qquad\hbox{and}\qquad
\norm{\cL\cT  A}{\Xkat}
\leq C.
\end{equation}
The proof of  estimates (\ref{f: required atomic}) 
will be given in Steps~II and III.

\emph{Step II:  proof of the first inequality 
in (\ref{f: required atomic})}.
Note that 
$\om\,\wh r$ has support in $[-3/4,3/4]$.
Define the functions $S_i$ as in (\ref{f: A0}) and
(\ref{f: Ai}).
Observe that, by 
(\ref{f: dec om hat r}),
\begin{equation}\label{f:S=sumSi}
\cS
= \sum_{i=0}^{d} S_i(\cD_1),
\end{equation}
where $d=[\!\![\log_{4}(3/R)+1]\!\!]$.
Denote by $B_i$ the ball with centre~$p$ and radius $(4^{i+1}+1)R$.
Since the support of the kernel of the operator $S_i(\cD_1)$ 
is contained in $\{(x,y): d(x,y) \leq 4^{i+1}R\}$ by the finite
propagation speed, the function $S_i(\cD_1)A$ is 
supported in $B_i$.

Now we check that $\cL S_i(\cD_1)A$ is in $\bigl(Q_{B_i}^k\bigr)^{\perp}$. 
By Proposition~\ref{p: canc II} it suffices to show 
that $\cL^{-k}\cL S_i(\cD_1)A$ is in $L^2_0(B_i)$. 
Now, $\cL^{-k}\cL S_i(\cD_1)A=S_i(\cD_1)\cL^{1-k}A$ and $\cL^{1-k}A$ 
is a constant multiple of 
a $H^1$-atom with support contained in $B(p,R)$
by Remark~\ref{rem: 12h atomi}. Thus the support of $S_i(\cD_1)\cL^{1-k} A$ 
is contained in $B_i$  and its integral  
over $M$ vanishes by Proposition~\ref{p: Mean}~\rmii.

Next, we claim that there exists a constant $C$, independent
of $A$, such that for  $i$ in $\{0,\ldots,d\}$
\begin{equation} \label{f: at size est I}
\norm{S_i(\cD_1)A}{2}
\leq C \,  
\, \mu(B_i)^{-1/2} \,4^{-i} 
\end{equation}
and
\begin{equation} \label{f: at size est II}
\norm{\cL S_i(\cD_1)A}{2}
\leq C\, 
\, \mu(B_i)^{-1/2}\,4^{-i}.
\end{equation}
Deferring momentarily the proof of the claim, we show that the 
first inequality in (\ref{f: required atomic}) follows from it. 
Indeed, by (\ref{f:S=sumSi}) and the triangle inequality,
$$
\norm{\cL\cS A}{\Xkat}
\leq C\, \sum_{i=0}^{d} \norm{\cL S_i(\cD_1) A}{\Xkat}.
$$
Now Remark~\ref{c: large atoms} and (\ref{f: at size est II})  
imply that 
$$
\begin{aligned}
\norm{\cL S_i(\cD_1) A}{\Xkat}
& \leq C \, \mu(B_i)^{1/2}\, \norm{\cL S_i(\cD_1) A}{2}  \\ 
& \leq C \,  4^{-i}. 
\end{aligned}
$$
Hence
$$
\norm{\cL\cS A}{\Xkat}
\leq C,
$$
as required to prove the first inequality in  (\ref{f: required atomic}).

To conclude the proof of Step~II it remains to prove 
(\ref{f: at size est I}) and (\ref{f: at size est II}). 
The function $S_0$ is bounded by Lemma~\ref{l: stima Runo} 
hence $S_0(\cD_1)$ is bounded on $\ld{M}$ by the spectral theorem, and
$$
\opnorm{S_0(\cD_1)}{2}\
\leq \norm{S_0}{\infty}.
$$
Since $S_{0}(\cD_1)A$ is supported in $B_0=B(p,5R)$, we have
$$
\norm{S_{0}(\cD_1)A}{2}
\leq 
 \opnorm{S_{0}(\cD_1)}{2}\, \norm{A}{2} 
\leq C\, R^{-n/2}.
$$
Furthermore, the integral of $S_0(\cD_1)A$ vanishes
by Proposition~\ref{p: Mean}~\rmii, 
so that $S_0(\cD_1)\,A$ is a constant multiple of an $H^1$-atom.

Denote by $k_{S_i(\cD_1)}$ the integral kernel of the
operator $S_i(\cD_1)$.  Observe that
$$
S_i(\cD_1)\, A(x)
= \int_{B(p,R)} A(y) \, \bigl[k_{S_i(\cD_1)}(x,y)
   -  k_{S_i(\cD_1)}(x,p)\bigr] \wrt \mu(y).
$$
By Minkowski's integral inequality
and the fact that the support of $S_i(\cD_1)\, A$
is contained in $B_i$, we have that
$$
\begin{aligned}
\norm{S_i(\cD_1)\, A}{2} 
& = \norm{S_i(\cD_1)\, A}{\ld{B_i}}  \\
& \leq \int_{B(p,R)} \mod{A(y)} \, I_{i}(y) \wrt \mu(y),
\end{aligned}
$$
where
$$
I_{i}(y) 
= \norm{k_{S_i(\cD_1)}(\cdot,y) 
- k_{S_i(\cD_1)}(\cdot,p)}{L^2(B_i)}
\quant y \in B(p,R).
$$
To estimate $I_{i}(y)$, we observe that 
$$
I_{i}(y)
\leq  d(y,p) \, \sup_{z\in M} \, \bignorm{\dest_2 
    k_{S_i(\cD_1)}(\cdot,z)}{\ld{B_{i}}},
$$
and, by 
(\ref{f: Ai}) and (\ref{f: decomposition R}),
$$
\dest_2 
    k_{S_i(\cD_1)}(\cdot,z)=\frac{1}{2\pi} \, \ir \phi^{\rho_i}(t)\, 
   P_{N}(\cO)(\om\wh{r})
   (t) \  \dest_2 
    k_{\cJ(t\cD_1)}(\cdot,z) \wrt t.
$$
Recall that $\phi^{\rho_i}$ is supported in $E_i
= \{ t\in \BR: 4^{i-1} R \leq \mod{t} \leq 4^{i+1}R \}$, 
that the support of $\om\wh{r}
   $
   is contained in $[-1,1]$ and that $d(p,y)<R$.
Then, by \cite[Prop.~2.2~\rmiii]{MMV1} (with $\cJ$
in place of $F$), there exists a constant $C$, independent of $i$
and $R$, such that
$$
\begin{aligned}
I_{i}(y)
& \leq C \, d(y,p) \, 
    \ir \phi^{\rho_i}(t)\, \mod{P_{N}\cO)
    (\om\wh{r})(t) }
\, \,
    \sup_{z\in M} \bignorm{\dest_2
     k_{\cJ (t\cD_1)}(\cdot,z)}{\ld{B_{i}}} \wrt t \\
& \leq C \, \norm{t P_{N}(\cO)
(\om\wh{r}) }
{\infty} \,R \,    
     \int_{E_{i}} \mod{t}^{-n/2-2}   \wrt t \\
& \leq C \, 
\, R \, (4^iR)^{-n/2-1} \,.
\end{aligned}
$$
Thus,
$$
\begin{aligned}
\norm{S_i(\cD_1)\, A}{2}
& \leq C\,
\, 4^{-i}\, 
      (4^i R)^{-n/2} \, \norm{A}{1} \\
& \leq C\, 
\, 
      4^{-i}\, \mu(B_i)^{-1/2}.
\end{aligned}
$$
This concludes the proof of (\ref{f: at size est I}).
Now we prove (\ref{f: at size est II}).  Recall that
$
\cL = \cD^2 + b\, \cI = \cD_1^2 + (b-\kappa^2) \, \cI.
$  
Therefore
\begin{equation} \label{f: stima LRunoa} 
\norm{\cL S_i(\cD_1) A}{2}
\leq \norm{\cD_1^2 S_i(\cD_1) A}{2} + \mod{b-\kappa^2} 
\, \norm{S_i(\cD_1) A}{2}.
\end{equation}
We first estimate $\norm{\cD_1^2 S_i(\cD_1) A}{2}$ 
when $i$ is in $\{1,\ldots,d\}$.  
Observe that 
$$
\cD_1^2\, S_i(\cD_1)
=\frac{2^{N-1}}{\sqrt{2\pi}} \ir \frac{\phi^{\rho_i}(t)}{t^2}\, 
P_N(\cO) (\om \, \wh r)(t)\, F (t \cD_1) \wrt t,
$$
where $F(\lambda)=\lambda^2 \cJ(\lambda)$.
Since the function $\la\mapsto F(t\la)$ is an even entire function of 
exponential type $\mod{t}$ and the support of 
$\phi^{\rho_i}$ is contained in the set  {$E_i$}, 
the  support of $F(t\cD_1)\, A$ is contained in $B_i$, 
by finite propagation speed. Thus
$$
F(t\cD_1)\, A(x)
= \int_{B(p,R)} A(y) \, \bigl[k_{F(t\cD_1)}(x,y)
   -  k_{F(t\cD_1)}(x,p)\bigr] \wrt \mu(y),
$$
and, by Minkowski's integral inequality,
$$
\begin{aligned}
\norm{F(t\cD_1)\, A}{2} 
& = \norm{F(t\cD_1)\, A}{\ld{B_i}} \\
& \leq \int_{B(p,R)} \mod{A(y)} \, I_{i}(y) \wrt \mu(y) \\
& \leq \norm{A}{1} \,\sup_{y\in B(p,R)}  I_{i}(y),
\end{aligned}
$$
where 
$$
I_{i}(y) 
=\norm{k_{F(t\cD_1)}(\cdot,y) 
- k_{F(t\cD_1)}(\cdot,p)}{\ld{B_i}}\quant y \in B(p,R).
$$
Observe that 
$$
I_{i}(y)
\leq  d(y,p) \,\, \sup_{z\in M} \, \bignorm{\dest_2 
    k_{F(t\cD_1)}(\cdot,z)}{\ld{B_i}}.
$$
Since
$
\sup_{\la \in \BR^+} (1+\la)^{N-2}\, \mod{F(\la)} < \infty
$
by the asymptotics of Bessel functions of the first kind and
$N-2 > n/2+1$ by assumption, 
we may use \cite[Prop.~2.2~\rmiii]{MMV1},
and conclude that 
$$
\sup_{z\in M} \, \norm{ \dest_2 k_{F(t\cD_1)}(\cdot,z)}{\ld{B_i}}
\leq C \, \mod{t}^{-n/2-1}
\quant t \in [-1,1].
$$
Therefore, since the support of $\phi^{\rho_i}$ 
is contained in the set {$E_i$},
$$
\begin{aligned}
\norm{\cD_1^2 S_i(\cD_1) A}{2}
& \leq C \, R \, \ir \frac{\phi^{\rho_i}(t)}{t^2}\, 
     \mod{P_N(\cO) (\om \, \wh r)(t)}\,  
    \sup_{z\in M} \bignorm{\dest_2
     k_{F(t\cD_1)}(\cdot,z)}{\ld{B_i}} \wrt t \\
& \leq C \,\norm{
t^{-1}P_N(\cO)
(\om \, \wh r)}{\infty} \, 
     \frac{R}{(4^{i}R)^{n/2+3}} \ir \phi^{\rho_i}(t) \, \mod{t} \wrt t \\
& \leq C \,     
     4^{-i} \, \mu(B_i)^{-1/2}
\quant i \in \{1,\ldots,d\}.
\end{aligned}
$$
Now, the inequality (\ref{f: at size est II}) for $i\in \{1,...,d\}$ follows directly  
fr{}om this, (\ref{f: stima LRunoa}) and (\ref{f: at size est I}).

Next we consider $\cL S_0(\cD_1)A$.  
Observe that $\cL S_0(\cD_1)A$ is supported in $B(p,5R)$, and that
$$
\begin{aligned}
\norm{\cL \,S_0(\cD_1)A}{2}
& \leq \bigopnorm{\cL \,S_0(\cD_1)}{2} \, \norm{A}{2} \\
& \leq \mu\bigl(B(p,R)\bigr)^{-1/2} \, 
\bigopnorm{\cL \,S_0(\cD_1)}{2} \\
& \leq C\, \mu\bigl(B(p,5R)\bigr)^{-1/2} \, 
\bigopnorm{\cL \,S_0(\cD_1)}{2}.
\end{aligned}
$$
To prove that $\cL\, S_0(\cD_1)$
is bounded on $\ld{M}$, with norm independent
of $R$ in $(0,1]$ observe that, by the spectral theorem and 
Lemma~\ref{l: stima Runo}
$$
\begin{aligned}
\bigopnorm{\cL\, S_0(\cD_1)}{2}
& \leq \sup_{\la \geq 0} \, (\la^2+b) \, \mod{S_0(\la)} \\
& \leq C,
\end{aligned}
$$
where $C$ is independent of $R$.
This concludes the proof of  
~(\ref{f: at size est II}), and of Step~II.  

\emph{Step III:  proof of the second inequality in~(\ref{f: required atomic})}.  
For each $j$ in $\{1,2,3,\ldots\}$, define~$\om_j$ by the formula
\begin{equation}\label{omj}
\om_j(t) = \om (t-j) + \om(t+j) \quant t \in \BR.
\end{equation}
Observe that $\sum_{j=1}^\infty \om_j=1-\om$
and that the support of $\om_j$ is contained in the set of all 
$t$ in $\BR$ such that $j-3/4\le\mod{t}\le j+3/4$.

In the rest of this proof, we write $\Om_{j,N}$ instead of 
$P_N(\cO)(\om_j\, \wh r)$.  
Observe that the support of $\Om_{j,N}$ 
is contained in $\set{t\in\BR:j-3/4\le\mod{t}\le j+3/4}$. 
Moreover, since $\wh{r}(t)={c}^{-1}\, \e^{-c\mod{t}}$ and 
$c>2\beta$ there exist  constants $C,\varepsilon>0$ such that
\begin{equation}\label{propOM}
\norm{\Om_{j,N}}{\infty}
\le C\, \e^{-2\beta j} 
\quant j\in \{1,2,\ldots\}.
\end{equation}

Define the function $T_j: \BR\to \BC$ by
\begin{equation} \label{f:ridef}
T_j(\la) 
= \ir  \Om_{j,N}(t)\, \cJ  (t \la)\wrt t \quant \la \in \BR.
\end{equation}
We may use the observation that $(m-\wh\om\ast m)\wh{\phantom a} 
= \sum_{j=1}^\infty \om_j\, \wh m$ 
and formula~(\ref{f: propertiesRiesz I}), and write
$$
\begin{aligned}
(m-\wh\om\ast m)(\la)
& = \frac{1}{2\pi} \ir  \bigl(1-\om(t)\bigr) \, \wh{r}(t)
      \, \cos (t\la) \wrt t \\ 
& = \sum_{j=1}^\infty T_j(\la). 
\end{aligned}
$$
Then, by the spectral theorem,
$$
\cT  A
=\sum_{j=1}^\infty T_j(\cD_1) A.
$$
Now we estimate $\norm{T_j(\cD_1) A}{2}$.  
By the asymptotics of $J_{N-1/2}$ \cite[formula (5.11.6), p.~122]{L} 
$$
\sup_{s>0} \mod{(1+s)^{N} \, \cJ  (s)} < \infty.
$$      
By Remark \ref{ultra} we may apply 
\cite[Proposition~2.2~\rmi]{MMV1}, since $N-1/2>(n+1)/2$,
and conclude that 
$$
\begin{aligned}
\norm{\cJ (t\cD_1) A}{2}
& \leq \norm{A}{1} \, \bigopnorm{\cJ (t\cD_1)}{1;2} \\
& \leq \sup_{y\in M}  \bignorm{k_{\cJ (t\cD_1)}(\cdot,y)}{2} \\
& \leq C\, \mod{t}^{-n/2}\, \bigl(1+\mod{t}\bigr)^{n/2-\de} 
\quant t \in \BR\setminus \{0\}.
\end{aligned}
$$
for some $\de>0$.
The function
$\cJ (t\cD_1)A$ is supported in $B(p,t+R)$ 
by Remark \ref{r: fps},
and has integral $0$ by Proposition~\ref{p: Mean}~\rmi.
%%%%%%%%%%%%%
Moreover
\begin{align*}
\norm{T_j(\cD_1)A}{2}
& \leq C\,   \ir  \mod{\Om_{j,N}(t)}  
     \,  \norm{\cJ  (t \cD_1)A}{2}  \wrt t \nonumber \\
& \leq C\, \int_{j-3/4}^{j+3/4} \mod{{\Om_{j,N}(t)}}  \,
         \mod{t}^{-n/2}\, \bigl(1+\mod{t}\bigr)^{n/2-\de} \wrt t \\
& \leq  C \, 
\e^{-2\beta \, j}  \quant j \in \{1,2,\ldots\}. \\
\end{align*}
By (\ref{f: volume growth}) there exist $\vep>0$ such that
$\e^{-2\be \, j}\leq C\, \mu\bigl(B(p,j+1)\bigr)^{-1/2} \,  \e^{-\vep \, j}$. 
Hence
\begin{equation} \label{f: est aj2}
\norm{T_j(\cD_1) A}{2}
\leq  C \,   
    \mu\bigl(B(p,j+1)\bigr)^{-1/2} \,  \e^{-\vep \, j} 
    \quant j \in \{1,2,\ldots\}. 
\end{equation}
Observe that, at least formally, 
$$
\cL \cT  A
=  \sum_{j=1}^\infty \, \cL \, T_j(\cD_1) A.
$$
To prove that the series converges in $\Xhat{M}$ we estimate $\norm{\cL T_j(\cD_1) A}{2}$.  Note that
\begin{equation} \label{f: stima Laj}
\norm{\cL T_j(\cD_1) A}{2}
\leq \norm{\cD_1^2 T_j(\cD_1) A}{2} + \mod{b-\kappa^2} 
\, \norm{T_j(\cD_1) A}{2}.
\end{equation}
We have already estimated $\norm{T_j(\cD_1) A}{2}$ in (\ref{f: est aj2}), 
so we concentrate on $\norm{\cD_1^2 T_j(\cD_1) A}{2}$. 
By (\ref{f:ridef}) and the spectral theorem
$$
\cD_1^2 T_j(\cD_1)
= \ir \Omega_{j,N}(t)\,F(t\cD_1) \frac{\wrt t}{t^2},
$$
where  $F(\la)=\la^2 \cJ(\la)$.
By using (\ref{propOM}), \cite[Proposition~2.2~\rmii]{MMV1}
and the fact that the support of $\Om_{j,N}$ 
is contained in $\set{t: j-3/4\le \mod{t}\le j+3/4}$, we obtain that 
there exist  constants $C$ and $\vep>0$ such~that
$$
\begin{aligned}
\norm{\cD_1^2T_j(\cD_1) A}{2}
& \leq C\,  \ir  \mod{\Om_{j,N}(t)}
     \,  \norm{F (t \cD_1)A}{2}  \frac{\wrt t}{t^2} \\
& \leq C\, \norm{A_{}}{1}\,  \ir \mod{\Om_{j,N}(t)} \, 
     \bigopnorm{F (t \cD_1)}{1;2} \frac{\wrt t}{t^2} \\
& \leq  C \,
 \,  \e^{-2\be \, j} \\
& \leq  C \, 
 \, 
     \mu\bigl(B(p,j+1)\bigr)^{-1/2} \, \e^{-\vep\, j} 
    \quant j \in {\{1,2,\ldots\}}.
\end{aligned}
$$ 
This estimate, (\ref{f: est aj2}) and (\ref{f: stima Laj}) then imply 
that 
\begin{equation} \label{f: Tja}
\norm{\cL \, T_j(\cD_1) A}{2}
\leq C\, 
  \, \mu\bigl(B(p, j+1)\bigr)^{-1/2}\,\e^{-\vep j}.
\end{equation}
Now, by (\ref{f:ridef}) we may write
\begin{equation} 
\begin{aligned}
\cL T_j(\cD_1)A  
& = \cL \ir  \Om_{j,N}(t)\, \cJ  (t \cD_1)A \wrt t \\
& = \cL^{k} \ir  \Om_{j,N}(t)\, \cJ  (t \cD_1)\cL^{1-k}A \wrt t \\
&= \cL^{k} a_j, 
\end{aligned}
\end{equation} 
where $a_j = \ir  \Om_{j,N}(t)\, \cJ  (t \cD_1)\cL^{1-k}A \wrt t$.
\par
The function $a_j$ is supported in $B(p,j+1)$, since $\cL^{1-k}A$
is in $L^2_0\big(B(p,R)\big)$ and the kernel of the operator 
$\int_{-\infty}^\infty  \Om_{j,N}(t)\, \cJ  (t \cD_1) \wrt t$ 
is supported in $\set{(x,y): d(x,y)\le j}$. Moreover, $\int_M a_j \wrt\mu=0$ by Proposition~\ref{p: Mean}~\rmii, and
\begin{equation} \label{f: est int TjA}
\norm{a_j}{2}
\leq C\, \opnorm{\cL^{-k}}{2}
  \, \mu\bigl(B(p, j+1)\bigr)^{-1/2}\,\e^{-\vep j},
\end{equation}
by (\ref{f: Tja}).  Hence $a_j$ is a multiple of an $H^1$-atom supported in $B(p,j+1)$.Then we may apply Lemma~\ref{l: economical decomposition I} to the function
$a_j$, and conclude
that $\cL T_j(\cD_1)A=\cL^k a_j$ is in $\Xhat{M}$, and that,  
by (\ref{f: Tja}),
$$
\begin{aligned}
\norm{\cL T_j(\cD_1)A}{\Xkat}   
&\leq C\, j\,\big(\mu(B(p,j+1)\big)^{1/2} \, \norm{\cL T_j(\cD_1)A}{2}\\
&\leq C\, 
  \, j\, \e^{-\vep j}.
\end{aligned}
$$
By summing over $j$, we see that
$$
\norm{\cL\cT A}{\Xkat}   
\leq C\, 
  \sum_{j=1}^\infty \, j\, \e^{-\vep j},
$$
thereby concluding the proof of Step~III and of the lemma.
\end{proof}


\begin{thebibliography}{GMMST}

\bibitem[Au]{Au} T. Aubin, Espaces de Sobolev sur 
les vari\'et\'es Riemanniennes, 
\emph{Bull. Sci. Math.}, \textbf{100} (1976), 149-173.

\bibitem[ACDH]{ACDH} P. Auscher, T. Coulhon, X.T.~Duong and
S. Hoffman,
Riesz transforms on manifolds and heat kernel regularity,
\emph{Ann.  Sc. \'Ec. Norm. Sup.} \textbf{37} (2004), 911--957.

\bibitem[AMR]{AMR} P. Auscher, A. McIntosh and E. Russ,
Hardy spaces of differential forms on Riemannian manifolds,
\emph{J. Geom. Anal.} \textbf{18} (2008), 192--248.

\bibitem[Br]{Br} R. Brooks, A relation between growth
and the spectrum of the Laplacian,
\emph{{Math. Z.}} \textbf{178} (1981), 501--508.

\bibitem[CMM1]{CMM1} A. Carbonaro, G. Mauceri and S. Meda, 
$H^1$, $BMO$ and singular integrals for certain metric measure
spaces,  \emph{ Ann. Sc. Norm. Super. Pisa Cl. Sci.}, \textbf{8} (2009), no. 3, p. 543-582

\bibitem[Ch]{Ch} I. Chavel, \emph{Riemannian geometry: 
a modern introduction}, Cambridge University Press, 1993.

\bibitem[CGT]{CGT} J. Cheeger, M. Gromov and M. Taylor,
Finite propagation speed, kernel estimates for functions of the
Laplace operator, and the geometry of complete Riemannian manifolds,
\emph{J. Diff. Geom.} \textbf{17} (1982), 15--53.

\bibitem[CW]{CW} R.R. Coifman and G. Weiss,  Extensions of Hardy
spaces and their use in analysis,
\emph{Bull. Amer. Math. Soc.} \textbf{83} (1977), 569--645.

\bibitem[CD]{CD} T. Coulhon and X.T. Duong,
Riesz transforms for $1\leq p \leq 2$,
\emph{Trans. Amer. Math. Soc.} \textbf{351} (1999), 1151--1169.


\bibitem[CMP]{CMP} M. Cowling, S. Meda and R. Pasquale,
Riesz potentials and amalgams,
\emph{Ann. Inst. Fourier Grenoble} \textbf{49} (1999), 1345--1367.


\bibitem[FeS]{FeS}
C. Fefferman and E.M.~Stein, $H^p$ spaces of several variables,
\emph{Acta Math.} \textbf{179} (1972), 137--193.

\bibitem[He]{He} E. Hebey, \emph{Sobolev Spaces on 
Riemannian Manifolds}, Lecture Notes 
in Mathematics \textbf{1635}, Springer Verlag, Berlin, 1996.


\bibitem[L]{L} N.N.~Lebedev, \emph{Special functions and their applications},
Dover Publications, 1972.

\bibitem[Lo]{Lo} N. Lohou\'e, Comparaison des champs de vecteus
et des puissances du laplacien sur une vari\'et\'e riemannienne
\`a courbure non positive
\emph{J. Funct. Anal.} \textbf{61} (1985), 164--201.

\bibitem[MRu]{MRu} M. Marias and E. Russ, $H^1$ boundedness of Riesz transforms
and imaginary powers og the Laplacian on Riemanian manifolds,
\emph{Ark. Mat.} \textbf{41} (2003), 115--132.


\bibitem[MMV1]{MMV1} G. Mauceri, S. Meda and M. Vallarino, 
Weak type $1$ estimates for functions of the Laplace--Beltrami 
operator on manifolds with bounded geometry,
{\tt arXiv:0811.0104 [math. FA.]}, to appear in \emph{Math. Res. Lett.}.

\bibitem[MMV2]{MMV2} G. Mauceri, S. Meda and M. Vallarino, 
Hardy type spaces on certain noncompact manifolds and applications, {\tt 	arXiv:0812.4209v1 [math.FA]}.

\bibitem[MMV3]{MMV3} G. Mauceri, S. Meda and M. Vallarino, 
Hardy type spaces on noncompact symmetric spaces,
in preparation. 


\bibitem[Gr] {Gr} A. Gry'goryan,  Estimates of heat kernels on 
Riemannian manifolds, in \emph{Spectral Theory and Geometry},
ICMS Instructional Conference Edinburgh 1988,
eds B.~Davies and Y.~Safarov, London Mathematical Society Lecture
Note Series \textbf{273}, Cambridge University Press, 1999.

\bibitem[Ru]{Ru}
E.~Russ, $H^1$--$L^1$ boundedness of Riesz transforms
on Riemannian manifolds and on graphs,
\emph{ Pot. Anal.} \textbf{14} (2001), 301--330.

\bibitem[St1]{St1} E.M. Stein, \emph{Topics in Harmonic Analysis Related to
the Littlewood--Paley Theory}, Annals of Math. Studies, No. \textbf{63},
Princeton N. J., 1970.

\bibitem[St2]{St2} E.M. Stein, \emph{Harmonic Analysis. Real variable
methods, orthogonality and oscillatory integrals}, Princeton Math. Series
No. \textbf{43}, Princeton N. J., 1993.

\bibitem[Str]{Str} R. S. Strichartz, Analysis of the Laplacian on a complete
 Riemannian manifold,  \emph{J. Funct. Anal.} \textbf{52} (1983), no. 1, 48--79. 

\bibitem[Ta]{Ta} M.E. Taylor,
$L^p$ estimates on functions of the Laplace operator,
\emph{Duke Math. J.} \textbf{58} (1989), 773--793.

\end{thebibliography}
\end{document}